\title{Sinkhorn-knopp balancing with transport constaints}
\date{}

\documentclass[twocolumn]{article}

\usepackage{url}
\usepackage{mathtools, cuted}
\usepackage{stfloats}
\usepackage{hyperref}
\usepackage{amsmath,amsfonts,amssymb}
\usepackage{bbm}
\usepackage{wrapfig}
\usepackage{draftwatermark}
\usepackage{algorithm}
\usepackage{algpseudocode}
\SetWatermarkText{}
\SetWatermarkScale{1}

\newtheorem{assumption}{Assumption}
\newtheorem{lemma}{Lemma}
\newtheorem{theorem}{Theorem}

\newtheorem{corollary}{Corollary}
\newtheorem{example}{Example}
\newtheorem{remark}{Remark}

\usepackage{bm}
\usepackage{mathptmx} 
\usepackage{boxedminipage}
\usepackage{graphicx}
\usepackage[T1]{fontenc}
\usepackage[english]{babel}
\definecolor{dgray}{gray}{0.6}
\definecolor{lgray}{gray}{0.8}
\newcommand{\ignore}[1]{}

\def\bbbz{{\mathchoice {\hbox{$\sf\textstyle Z\kern-0.4em Z$}}
{\hbox{$\sf\textstyle Z\kern-0.4em Z$}}
{\hbox{$\sf\scriptstyle Z\kern-0.3em Z$}}
{\hbox{$\sf\scriptscriptstyle Z\kern-0.2em Z$}}}}

\newcommand{\ba}{\begin{array}}
\newcommand{\ea}{\end{array}}
\newcommand{\beq}{\begin{equation}}
\newcommand{\eeq}{\end{equation}}
\newcommand{\beqy}{\begin{eqnarray}}
\newcommand{\eeqy}{\end{eqnarray}}
\newcommand{\beqyn}{\begin{eqnarray*}}
\newcommand{\eeqyn}{\end{eqnarray*}}
\newcommand{\bi}{\begin{itemize}}
\newcommand{\ei}{\end{itemize}}
\newcommand{\bseq}{\begin{subeqnarray}}
\newcommand{\eseq}{\end{subeqnarray}}

\newcommand{\bex}{\begin{example}}
\newcommand{\eex}{\end{example}}
\newcommand{\bexer}{\begin{exercise}}
\newcommand{\eexer}{\end{exercise}}
\newcommand{\bmct}{\begin{fact}}
\newcommand{\efct}{\end{fact}}

\def\qed{ \rule{.1in}{.1in}}

\usepackage{etoolbox}
\usepackage{authblk}

\patchcmd{\IEEEproofindentspace}{2\parindent}{-2pt}{}{}

\usepackage{caption}
\usepackage{subcaption}

\begin{document}

\title{Sinkhorn-Knopp balancing with generalised martingale-type constraints}

\author[1]{Abigail Langbridge}
\author[2]{Martin Corless}
\author[1]{Robert Shorten}
\affil[1]{Dyson School of Design Engineering, Imperial College London, UK}
\affil[2]{School  of Astronautics and Aeronautics, Purdue University, West Lafayette, Indiana, USA}


\maketitle

\begin{abstract}
We consider the problem of optimally distributing resources from a set of suppliers to a set of consumers in the presence of general transportation constraints, and including heterogeneous flexibility in the marginal constraints. Such problems frequently arise in a variety of practical settings; for example, in the context of sharing economy applications, where one is not only interested in the transportation plan, but also its realisation, and in other problems that involve the study of martingales. Our principal contribution in this paper is to consider a generalisation of the classic entropically regularised Optimal Transport formulation in which such problems can be solved with a Sinkhorn algorithm.
In particular, we present provably convergent Sinkhorn-like algorithms for solving this class of problems, and provide examples to both illustrate the utility of our approach as well as its efficacy.
\end{abstract}

\section{Introduction}

In this paper we consider the problem of optimally distributing resources from a set of suppliers to a set of consumers in the presence of heterogeneous marginal and transportation constraints. Such problems arise in classical economic settings; in particular, in sharing economy settings \cite{shorten}, when suppliers, consumers, and agents cooperate to solve a complex resource reallocation task (for example, in the gig economy). In problems of this type, a typical goal is to reallocate a resource in an efficient manner, subject to individual agents meeting some minimum income requirement. A contemporary and very important example where such a problem arises is in the realisation of equality of opportunity constraints across the agent sub-groups tasked to realise the transportation plan. More specifically, consider the problem of allocating passengers to drivers for on-demand ride hailing. As studied in \cite{cook2021gender}, a gender pay gap persists in this type of gig work despite its inherent flexibility. One key contribution to the pay gap is the location of work, with male drivers more likely to live and work closer to lucrative locations, such as city centres and high-crime neighbourhoods. To overcome this, we propose an optimal transport method which admits a novel constraint on the coupling between passengers and drivers to equalise the total utility of male and female drivers in the system, while minimising the waiting time for passengers. Our approach also accounts for unbalanced couplings, where there may be more drivers than passengers, and can be solved efficiently using Bregman projections.

Optimal transport (OT) is a fundamental methodological framework underpinning a wide variety of problems in economics, control and in machine learning \cite{gal,gan2023optimal,montesuma2024recent}. 
Notable recent successes of the technology have been reported in domain adaptation and transfer learning applications \cite{pey,kolkin2019style,peterson2021transfer}
and in the related areas of correcting and detecting algorithmic bias in machine learning algorithms \cite{gor,abi,quan}. Despite this recent interest, the field of Optimal Transport (OT) is an old topic, dating back to the the work of Monge in the 18th century \cite{monge1781memoire}, and more recently to the work of the Russian mathematician Kantorovich in the 20th century \cite{kantorovich1942translocation}.

The classical formulation of OT refers to a class of problems that describe the transportation of a finite resource from a bunch of sources  to a bunch of targets in an optimal manner. In the discrete setting, the solution of this problem gives rise to a linear program, which is known to scale poorly. Consequently, much recent work in OT has focussed on the development of efficient algorithms for solving an appropriately strictly {convexified} (but approximate) OT problem. The celebrated Sinkhorn-Knopp (SK) algorithm has arisen in this context and is obtained by adding a so-called entropic regularisation term to the OT problem, thereby realising a strictly convex cost for which efficient algorithms can be used \cite{Cut1,Bregman1967}.

This formulation can be extended with additional constraints to permit more expressive OT formulations including 
multi-marginal OT problems \cite{gangbo1998optimal,pass2015multi},
problems where the optimal couplings are martingales \cite{galichon2014stochastic,guo2019computational}, and those with bespoke constraints such as the sparsity constraint proposed by \cite{liu2022sparsity} or the zero-transport entries in \cite{martin}. However, many of these works are not provably convergent, and they require the marginal constraints to be strictly enforced. Therefore they are not appropriate for solving \textit{unbalanced} \cite{benamou2003numerical,chizat2018unbalanced} OT problems, where the marginals may not be of equal mass, or not all mass need be transported. 
Our interest is to extend this recent line of work to situations where more general heterogenous constraints are enforced on the marginals, and on the transportation plan itself. Such problems arise in sharing economy applications where suppliers and consumers have differing contracts, and where the realisation of the transportation plan itself gives rise to constraints.

In this work, we present a generalised, provably convergent algorithm for solving constrained optimal transport problems. We show that our approach is robust to selectively- and completely-relaxed marginal constraints, as well as general linear constraints on the coupling. In Section \ref{sec:applications}, we explore the application of this algorithm to the fair allocation of work in a ridesharing example, and the allocation of insufficient supply in an energy prosumer market, where some consumers are flexible.

\subsection{Specific contributions}
The work presented in this paper significantly extends previous work in several regards. Specifically, the novel contributions of our work as follows:
\begin{itemize}
    \item We develop provably convergent algorithms for solving optimal transport problems with arbitrary linear constraints, including hard zero constraints on elements of the plan.
    \item We also allow for arbitrary compliance and non-compliance with the marginal constraints on the transport plan to model flexibility in the supply and/or demand.
    \item We show that when the optimal plan is on the constraint boundary, all allowable plans must also be on the same boundary.
    \item We illustrate the application of this approach to several resource allocation problems.
\end{itemize}

\section{Relaxed OT with transportation constraints}
\label{sec:OTzero}
Our objective in this section is to formulate a generalized  OT problem with a general set of constraints. In doing this we formulate an OT problem to achieve the following:
\begin{itemize}
\item[1] To constrain individual elements of the transportation plan to be zero, for example to forbid certain supplier-consumer relationships or to reflect missing infrastructure;

\item[2] To encode some \textit{ideal} transport plan, through which the problem can be regularised;

\item[3] To formulate a set of general constraints on the transportation plan. These constraints allow several interesting case studies to be explored, including equalising opportunity for subgroups charged with realising the plan, or allowing heterogeneous flexibility in resource allocation.

\end{itemize}

We proceed as follows. Consider an optimal transport (OT) problem in which there are a finite number, $m$, of {\em source agents}  (that is, source states), $i=1,2, \dots, m$,  and a finite number, $n$, of {\em target agents} (that is,\ target states),  $j= 1,2, \dots, n$.
Suppose each source agent, $i$, has {\em capacity} or {\em mass},  $\tilde{u}_i  >0$,  and each target agent, $j$, has a desired capacity or mass, $\tilde{v} _j >0$.
We wish to transport the masses of the source agents to the target agents.
 {\em Let
$
t_{ij} \ge 0
$
be the mass or capacity that source agent, $i$,  transports to  target agent, $j$.}
We will refer to  the $m \times n$ matrix, $T=\left\{t_{ij}\right\}$, as the {\em transport plan} or {\em transport map}.

Suppose that  there is a cost, $c_{ij} \in \mathbb{R}$, of moving a unit mass from source, $i$, to
target,  $j$.
Then, the total cost of transport is
\begin{equation}
\label{eq:cost0}
\sum_i^n \sum_{j}^m c_{ij} \; t_{ij} 
\end{equation}

In the {\em standard OT problem}, each source agent transports all of its capacity, that is,
\begin{align}
\label{eq:rowconstraints}
\sum_{j=1}^nt_{ij}  &= \tilde{u}_i  \qquad i=1, 2, \dots, m 
\end{align}
or, $T {\bm 1} = \tilde{u}$, where   ${\bm 1}$ is a vector of ones of appropriate length.
Also,
each target agent receives its desired capacity, that is,
\begin{align}
\label{eq:colconstraints}
\sum_{i=1}^mt_{ij}  &= \tilde{v}_j  \qquad j=1, 2, \dots, n
\end{align}
or, $T' {\bm 1} = \tilde{v}$.
When constraints   \eqref{eq:rowconstraints}  and \eqref{eq:colconstraints} are satisfied we have 
\[
\sum_{i=1}^m\sum_{j=1}^n t_{ij} =\sum_{i=1}^m \tilde{u}_i = \sum_{j=1}^n \tilde{v}_j
\]
that is, 
\begin{equation}
{\bm 1}'\tilde{u} = {\bm 1}'\tilde{v}
\label{eq:bal}
\end{equation}
Problems in which \eqref{eq:bal} is satisfied are called {\em balanced} OT problems.
In {\em  unbalanced OT (UOT)} \eqref{eq:bal}  is not satisfied; hence one cannot  simultaneously satisfy  \eqref{eq:rowconstraints}  and \eqref{eq:colconstraints}.

\subsection{Encoding zero constraints}
Now suppose that a  specific source agent, $i$, can never  transport to a specific target agent, $j$, so that 
 $t_{ij}= 0$.
We let $\mathcal {Z}$ be the complement of the  set of these prior non-transporting  index pairs, that is,
\begin{equation}
\label{eq:U}
t_{ij}= 0  \text{ if } (i ,j) \notin \mathcal {Z} 
\end{equation}
Then the set of allowable transport plans is given by
\begin{align}
\mathcal{T} \equiv \left\{ T \in \mathbb{R}^{m\times n} :  \begin{array}{ll} t_{ij}  \ge 0 & \text{if }  (i,j) \in  \mathcal{Z} \\ 
t_{ij} = 0 & \text{if }   (i,j) \notin  \mathcal{Z} \end{array}  \right\}
\label{eq:Tset}
\end{align}
We assume that for every $i$, that there is at least one $(i, j) \in \mathcal{Z}$
and for every $j$, there is also at least one $(i, j) \in \mathcal{Z}$;
that is  if 
$t_{ij} >0$ for all $(i, j) \in \mathcal{Z}$
then, {\em $T$ has   no row or column constrained to be all zeros}.

If every source agent can transport to every target agent,  and the balanced condition \eqref{eq:bal} holds, then the marginal constraints \eqref{eq:rowconstraints} and \eqref{eq:colconstraints} are feasible (that is,  there is at least one solution in $\mathcal{T}$, for example, 
$T \equiv  \tilde{u}\tilde{v}'/{\bm 1}'\tilde{v}$).
 However, if some  source agent cannot transport to some target agent
 then, \eqref{eq:rowconstraints} and \eqref{eq:colconstraints}  may not be both feasible (that is, there may not be any solution in $\mathcal{T}$). A simple example to illustrate this is given  by  $m=n=2$ and  $\mathcal{Z} = \{(1,2),(2,1)\}$: 
\begin{equation}
T=\left(\begin{array}{cc}
0 & t_{12} \\
 t_{21} & 0
\end{array}
\right)
\qquad  \tilde{u} = 
 \tilde{v} =\left(\begin{array}{c} 2\\ 1
\end{array}
\right)
\label{eq:exsimple}
\end{equation}

\subsection{Encoding an ideal (reference) plan}
Suppose that it is desirable for $T$ to be close to some {\em ideal (desired) transport plan}, $\tilde{T}$
with 
 $\tilde{t}_{ij} >0$ for $(i,j) \in \mathcal{Z}$.
We therefore  modify 
the cost \eqref{eq:cost0} to
\begin{equation}
\label{eq:newJ}
\sum_{(i,j) \in \mathcal{Z}} c _{ij} t_{ij} +\gamma_0 KL(T|\tilde{T})
\end{equation}
where  $\gamma_0 >0$ is a {pre-assigned  constant, whose size reflects the importance of the transport plan being close to the ideal plan}, and the {\em Kullback-Leibler ($KL$) divergence} of $T$ from $\tilde{T}$ is defined as
\begin{align}
KL(T| \tilde{T}) \equiv   \sum_{(i,j) \in \mathcal{Z}}  kl(t_{ij}|\tilde{t}_{ij})	
\label{eq:KL}
\end{align}
where,  for any two scalars, $t\ge 0$ and $\tilde{t}>0$, we define
\begin{equation}
kl(t|\tilde{t}) \equiv
\left\{\begin{array}{ccc}
\tilde{t}	&\text{if}	& t = 0	\\
\displaystyle{ t \log \left(\frac{t}{\tilde{t}}\right)} -t + \tilde{t}	& \text{if}	& t >0
\end{array}
\right.
\end{equation}
One may readily show that,  for all $t\ge0$ and $\tilde{t}>0$,  $kl(t|\tilde{t}) \ge 0$; also $kl(t|\tilde{t}) =0$ iff  $t=\tilde{t}$. Hence, for all $T\in \mathcal{T}$,  $KL(T|\tilde{T}) \ge 0$ and $KL(T|\tilde{T}) =0$ iff $T=\tilde{T}$.

\begin{remark}
The problem of minimizing \eqref{eq:cost0} subject to linear constraints is a linear programming problem and is computationally expensive for large problems. So a common approach is to add a regularizing term of the form $\gamma_0 KL(T|\tilde{T})$ 
with $\gamma_0>0$ small to obtain a less computationally burdensome problem.
\end{remark}

Cost \eqref{eq:newJ} can be expressed as 
$
\gamma_0 KL(T|K) + \sum_{(i,j) \in \mathcal{Z}} \gamma_0 (\tilde{t}_{ij} - k_{ij})
$
where
\begin{align}
k_{ij} & \equiv \left\{
\begin{array}{ccl}
\tilde{t}_{ij}\exp(-c_{ij}/\gamma_0) \qquad & \text{if} \qquad & (i,j) \in \mathcal{Z}	\\
0	\qquad & \text{if} \qquad & (i,j)  \notin \mathcal{Z}
\end{array}
\right.	
\label{eq:K}
\end{align} 
Therefore,  we replace cost \eqref{eq:newJ} with
$KL(T|K)$

\subsection{Encoding generalised transportation constraints}
Consider the regularised OT problem with linear constraints:
\begin{equation}
\label{eq:GenOpt2}
\fbox{$\displaystyle
\min_{T \in \mathcal{T}} L(T)\quad  \text{subject to} \quad \eqref{eq:GenConst} 
$}
\end{equation}
\begin{equation}
L(T) =KL(T|K) + \sum_{l=1}^{N_1} \gamma_l kl\left( \sum_{(i,j) \in {\mathcal Z}} 
\left. a^l_{ij} t_{ij}\right|\tilde{b}^l  \right)
\end{equation}
\begin{equation}
\label{eq:GenConst}
\sum_{(i,j) \in \mathcal Z} a^l_{ij}t_{ij} = \tilde{b}^l
\qquad l=N_1+1, N_1+2, \cdots, N
\end{equation}
where
$a^l_{ij} \ge 0$, and  $\tilde{b}^l >0$ for $l=1, 2, \cdots, N_1$.

Note that this formulation encompasses several interesting classes of constraints, including: capacity constraints on individual elements of the plan $t_{ij}$; constraints on the total mass supplied or consumed under the plan, and; \textit{martingale constraints} which constrain the conditional expectation of the target and are popular in financial modelling \cite{galichon2014stochastic}.
%
%
The martingale property can be formulated as
\begin{equation}
\sum_{j=1}^nt_{ij}y_j =  \sum_{j=1}^nt_{ij} x_i \qquad i =1,2, \dots, m
\end{equation}
    for some $x_i, y_j \in \mathbb{R}$.
    This can be expressed as
\begin{equation}
\sum_{j=1}^n(y_j-x_i)t_{ij} =0 \qquad i =1,2, \dots, m
\label{eq:martingale_constraint}
\end{equation}
which is of the form of \eqref{eq:GenConst}. 


\subsubsection{A special case}
Here we consider OT problems in which {\em some of the row constraints \eqref{eq:rowconstraints} are satisfied}, that is,
\begin{align}
\label{eq:rowconstraints1}
\sum_{j=1}^nt_{ij}  = \tilde{u}_i  \qquad i \in \mathcal{I}
\end{align}
where $\mathcal{I}$ is a  subset of $\{1,2,\dots, m\}$ and
 {\em some of  the column constraints \eqref{eq:colconstraints} are satisfied}, that is,
\begin{align}
\label{eq:colconstraints1}
\sum_{i=1}^m t_{ij} = \tilde{v}_j  \qquad j \in \mathcal{J} 
\end{align}
where $\mathcal{J}$ is a  subset of $\{1,2,\dots, n\}$.
This is a generalization of the situation considered in \cite{martin}; there all the row constraints were enforced and none of the column constraints were enforced.

We  now add KL-type penalty terms to the cost function, 
which are measures of non-compliance with the standard  row and column constraints \eqref{eq:rowconstraints} and \eqref{eq:colconstraints}.
We therefore consider a  cost function given by
\begin{equation}
\label{eq:lt}
\begin{aligned}
L(T) &= KL(T|K)
+ \sum_{i \notin \mathcal{I}}  \gamma_{1i} kl\left(\left. \sum_{j=1}^n  t_{ij} \right| \tilde{u}_i\right) \nonumber\\
&+ \sum_{j \notin \mathcal{J}} \gamma_{2j} kl\left(\left. \sum_{i=1}^m  t_{ij} \right| \tilde{v}_j \right)
\end{aligned} 
\end{equation}
where 
$\gamma_{1i}, \gamma_{2j}>0$.
Since we do not necessarily enforce all the standard row and  column constraints, \eqref{eq:rowconstraints} and \eqref{eq:colconstraints},  
we need $\gamma _{1i}, \gamma_{2j}>0$ to penalize non-enforcement of these constraints. 
So here we consider the following problem:
\begin{equation}
\label{eq:GenOpt}
\fbox{$\displaystyle
\min_{T \in \mathcal{T}} L(T)\quad  \text{subject to} \quad \eqref{eq:rowconstraints1} \text{ and } \eqref{eq:colconstraints1}
$}
\end{equation}
Clearly this is a special case of the optimization problem \eqref{eq:GenOpt2}.


\section{Algorithms to solve  OT problems }
\label{sec:algore's_favorite_dance}
We now present the main results of the paper. These results yield iterative scaling algorithms that produce a sequence, $\{T(k)\}$, converging to the optimal transport plan, $T^*$, for 
\eqref{eq:GenOpt2}. 
In particular, two algorithms are presented. Algorithm 1 is obtained using results in \cite{Bregman1967}.
In order to obtain an algorithm which is a generalization of the Sinkhorn-Knopp algorithm, we introduce new parameters to obtain Algorithm 2.
Algorithms 3 and 4 specialise Algorithms 1 and 2 to the OT problem with relaxed marginal constraints defined in \eqref{eq:GenOpt}.

{\bf Algorithm 1.}
Initialize with 
\begin{equation*}
   t_{ij}(0) = k_{ij}\quad (i,j) \in {\mathcal Z} \qquad b^l(0) = \tilde{b}^l, \;\; l=1,2,\cdots,  N_1
\end{equation*} 
Iterate  for $k= 0, 1 \dots:$

Let $l = k\!\bmod N +1$.
\newline
If $1 \le l \le N_1$,
let
\begin{equation}
\begin{aligned}
\label{eq:tij(k+1)1}
t_{ij}(k+1) &= c_{l}(k+1)^{a^l_{ij}}  t_{ij}(k) \quad (i, j) \in {\mathcal Z}\\
b^l(k+1) &= c_{l}(k+1)^{-1/\gamma_l} b^l(k)
\end{aligned}
\end{equation}
where $c_{l}(k+1) >0$ satisfies
\begin{equation}
\label{eq:cl(k+1)1}
 \sum_{(i,j) \in {\mathcal Z}} a^l_{ij} t_{ij}(k) c_{l}(k+1)^{a^l_{ij}}   -b^l(k)c_{l}(k+1)^{-1/\gamma_l}  =0
\end{equation}
If  $l=N_1+1, N_1+2, ,\dots,  N$, let
\begin{equation}
\label{eq:tij(k+1)2}
t_{ij}(k+1) = c_{l}(k+1)^{a^l_{ij}}  t_{ij}(k) \quad (i, j) \in {\mathcal Z}\\
\end{equation}
where $c_{l}(k+1) >0$ satisfies
\begin{equation}
\label{eq:cl(k+1)2}
 \sum_{(i,j) \in {\mathcal Z}} a^l_{ij}  t_{ij}(k) c_{l}(k+1)^{a^l_{ij}}  = \tilde{b}^l
\end{equation}



%

The following theorem provides the  first main result of this paper. A proof is provided in Section \ref{sec:Proof}.
\begin{theorem}
\label{th:Genal1}
Suppose the optimization problem \eqref{eq:GenOpt2} is feasible, 
Then the sequence $\{T(k), b(k)\}$  generated by Algorithm 1 converges to a limit, $(T^*, b^*)$, 
and $T^*$ is a minimizer for  the optimization problem   given by \eqref{eq:GenOpt2}.
\end{theorem}

{In order to obtain an algorithm for comparison with the Sinkhorn-Knopp algorithm,} we  introduce scaling parameters,
\begin{align*}
d_l(0) &= 1 	\qquad \qquad \qquad l=1,2,\dots, N\\
\quad  d_l(k+1) &=c_l(k+1)d_l(k) \quad  l=l_k  \\
\quad  d_l(k+1) &=d_l(k)  \, \, \qquad \qquad l=1,2,\dots, N,  l\neq l_k
\end{align*}

{\bf Algorithm 2.}
Initialize with
 \[
d_l(0) = 1\qquad l=1,2,\dots, N 
\]
Iterate for $k=0,1, \cdots$:

With  $l_k = k\bmod N +1$, let
\begin{equation}
\label{eq:dlkp1}
d_l(k+1) =d_l(k)  \qquad  l=1,2,\dots, N \qquad  l\neq l_k
    \end{equation}
 If $1 \le l_k \le N_1$, for $l=l_k$,
solve
\begin{equation}
\begin{aligned}
\label{eq:dlGen}
 &\sum_{(i,j) \in {\mathcal Z}} {a^l_{ij}}k_{ij}   \prod_{s=l, s\neq l}^Nd_{l}(k)^{a^s_{ij}} d_{l}(k+1)^{-1/\gamma_l} 
 - \tilde{b}^l   d_{l}(k+1)^{a^l_{ij}} =0
 \end{aligned}
\end{equation}
for $d^l(k+1)$.

If $N_1+1 \le l_k \le N$, for $l=l_k$,
solve
\begin{equation}
\begin{aligned}
\label{eq:dlGen}
 &\sum_{(i,j) \in {\mathcal Z}} {a^l_{ij}}k_{ij}   \prod_{s=l, s\neq l}^Nd_{l}(k)^{a^s_{ij}} d_{l}(k+1)^{-1/\gamma_l} 
 = \tilde{b}^l   
 \end{aligned}
\end{equation}
for $d_l(k+1)$.

Let
\begin{equation}
\begin{aligned}
t_{ij}(k+1) = \prod_{l=1}^Nd_{l}(k+1)^{a^l_{ij}} k_{ij}  \\
\end{aligned}
\end{equation}

Using Algorithm 1, we obtain an algorithm for the problem in \eqref{eq:GenOpt};
the details are in Section \ref{sec:3from1}.

{\bf Algorithm 3.}
Initialize
 \[
 T(0)= K, \; u(0) = \tilde{u}, \; \,v(0) = \tilde{v}
 \]
Iterate for $k=0,1, \dots$,
 \begin{equation}
 \begin{aligned}
c_{1i}(k+1) 	&=  \frac{\tilde{u}_i}{\Sigma_{j=1}^n t_{ij}(k)}	\qquad \qquad \qquad i \in \mathcal{I}	\label{eq:c_1Iterate}	\\
 c_{1i}(k+1)  &= \left(\frac{u_i(k)}{\Sigma_{j=1}^n t_{ij}(k)}\right)^{\frac{\gamma_{1i} }{1+\gamma_{1i}}}	
 \qquad i \notin \mathcal{I}\\
 u_{i}(k+1) &= c_{1i}(k+1)^{-1/\gamma_{1i}} u_i(k)
   \qquad i \notin \mathcal{I} \\
\\
 c_{2j}(k+1) 	&=  \frac{\tilde{v}_j}{\Sigma_{i=1}^m c_{1i}(k+1) t_{ij}(k)}	
  \qquad j \in \mathcal{J} \\
 c_{2j}(k+1)  &= \left(\frac{v_j(k)}{\Sigma_{i=1}^m c_{1i}(k+1)t_{ij}(k)}\right)^{\frac{\gamma_{2j} }{1+\gamma_{2j}}}
   \qquad  j \notin \mathcal{J} \\
  v_{j}(k+1) &= c_{2j}(k+1)^{-1/\gamma_{2j}} v_j(k)	
   \qquad j \notin \mathcal{J} \\
  \\
   t_{ij}(k+1)   &= c_{1i}(k+1)t_{ij}(k)c_{2j}(k+1)	\qquad (i,j) \in \mathcal{Z}
 \end{aligned}
 \end{equation}

Here, we introduce two scaling parameters,
 $d_{1i}(k)$ and $d_{2j}(k)$, defined by
\begin{align}
d_{1i}(k+1) &\equiv c_{1i}(k+1)d_{1i}(k) \quad d_{1i}(1) = c_{i1}(1)	\\
d_{2j}(k+1) &\equiv c_{2j}(k+1)d_{2j}(k) \quad d_{2j}(0) = 1
\end{align}
Then
\begin{align}
   u_{i}(k+1) &= d_{1i}(k+1)^{-1/\gamma_{1i}} \tilde{u}_i  \quad\qquad i \notin \mathcal{I}	\\
  v_{j}(k+1) &= d_{2j}(k+1)^{-1/\gamma_{2j}} \tilde{v}_j 	\quad \qquad  j \notin \mathcal{J}	
    \label{eq:v_jIterate2}\\
    t_{ij}(k+1)   &= d_{1i}(k+1)t_{ij}(k)d_{2j}(k+1)	\;\; \,(i,j) \in \mathcal{Z}
 \label{eq:t_ijIterate2} 
\end{align}
and the sequence, $\{T(k)\}$, obtained from Algorithm 3, can then also be obtained from the following algorithm.\\


 {\bf Algorithm 4.}
 Initialize
 \[
  d_{2j}(0) = 1
\]
 Iterate for $k= 0,1, \dots$, 
 \begin{equation}
\begin{aligned}
 d_{1i}(k+1)  &=  \frac{\tilde{u}_i}{\Sigma_{j=1}^n k_{ij}d_{2j}(k)}\qquad \qquad i \in \mathcal{I}			
	\\
 d_{1i}(k+1) &= \left(\frac{\tilde{u}_i}{\Sigma_{i=1}^n k_{ij}d_{2i}(k)}\right)^\frac{\gamma_{1i}}{1+\gamma_{1i}}
\qquad i \notin \mathcal{I}
	\\
	d_{2j}(k+1) &= \frac{\tilde{v}_j}{\Sigma_{i=1}^m d_{1i}(k+1)k_{ij}}	\qquad \qquad j \in \mathcal{J}	\\
d_{2j}(k+1) &= \left(\frac{\tilde{v}_j}{\Sigma_{i=1}^m d_{1i}(k+1)k_{ij}}\right)^\frac{\gamma_{2j}}{1+\gamma_{2j}}
 \quad	j \notin \mathcal{J}\\
\\
 t_{ij}(l+1) &= d_{1i}(k+1)k_{ij}d_{2j}(k+1)\qquad (i,j) \in \mathcal{Z}
 \end{aligned}
 \end{equation}


\section{Interior and non-interior optimizers}

\subsection{A consequence of a non-interior optimizer}
In much of the OT problems previously considered in the literature, the  minimizer $T^*$ is in the interior of ${\mathcal T}$.
However, when one has hard zero constraints of some of the elements of $T$,
it is possible that the minimizer $T^*$ is not in the interior of ${\mathcal T}$, that is, $t^*_{ij} =0$ for some $(i,j) \notin {\mathcal Z}$.
To see this, consider the original OT  problem in
\eqref{eq:GenOpt} with 
\begin{equation}
T=\left(\begin{array}{cc}
t_{11} & t_{12}\\
t_{21}& 0
\end{array}
\right)
\quad \mbox{with} \quad  \tilde{u} = \tilde{v} =\left(\begin{array}{c} 1\\ 1
\end{array}
\right)
\label{eq:exsimple2}
\end{equation}
and ${\mathcal I}= {\mathcal J} = \{1, 2\}$.
There is only one matrix
\[
\left(\begin{array}{cc}
0& 1\\
1& 0
\end{array}
\right)
\]
satisfying constraints  \eqref{eq:rowconstraints1} and \eqref{eq:colconstraints1}, hence it is optimal. Here $t^*_{11} = 0$.

Considering the general OT problem in \eqref{eq:GenOpt2}, we will show that, if for any $(i,j) \in {\mathcal Z}$ there is a matrix $T$ satisfying the constraints 
\eqref{eq:GenConst}
with $t_{ij} >0$ then $t^*_{ij}>0$, or equivalently, if $t^*_{ij}=0$ then $t_{ij} = 0$ for all matrices satisfying the constraints.

First, we obtain a more general result. To this end, let
\[
\mathcal{X} = \mathbb{R}^q_{++} =\left\{ x \in \mathbb{R}^q: x_i >0 \text{ for } i =1,2, \dots, q \right\}
\]
Then, the closure of ${\mathcal X}$ is
\[
\bar{\mathcal X} = \mathbb{R}^q_{+} =\left\{ x \in \mathbb{R}^q: x_i \ge0 \text{ for } i =1,2, \dots, q \right\}
\]
\label{sec:propmin}
\begin{lemma}
\label{lem:lem1}
Suppose  ${\mathcal C}$ is a convex subset of $\mathbb{R}^q$ and 
$x^*$ minimizes 
\begin{equation}
\label{eq:genCost1}
f(x)= \sum_{i=1}^q\gamma_i kl(x_i|\tilde{x}_i)
\end{equation}
over ${\mathcal C}\cap\bar{\mathcal X}$
where $\tilde{x}_i \in {\mathcal X}$ and $\gamma_i >0$ for $i=1,2,\dots, q$.
If  $x^*_{i} =0 $ for some $i$ then, $x_i = 0$ for all $x \in {\mathcal C} \cap \bar{\mathcal X}$.
\end{lemma}

\begin{proof}
Let  $x^*$ minimize 
$f$
over $ {\mathcal C} \cap \bar{\mathcal X}$. 
Suppose, on the contrary that, for some $i$, $x^*_i=0$ and $x_i >0$ for some $x \in {\mathcal C} \cap \bar{\mathcal X}$
 and let 
\begin{align*}
\mathcal {I}_0 &= \{ i \in \{1,2,\dots, q\}: x^*_i =0 \mbox{ and } x_i = 0\}	
\\
\mathcal {I}_1 &= \{ i \in \{1,2,\dots, q\}: x^*_i >0 \}	\\
\mathcal {I}_2 &= \{ i \in \{1,2,\dots, q\}: x^*_i =0 \mbox{ and } x_i >0\}	
\end{align*}
Then $\mathcal {I}_2$ is non-empty.

Consider  any  $\lambda \in (0, \frac{1}{2}]$   and  let $y=(1-\lambda)x^* + \lambda x$.
Since $x^*, x \in   {\mathcal C} \cap \bar {\mathcal X}$ and  ${\mathcal C} \cap \bar {\mathcal X}$ is convex, $y \in {\mathcal C} \cap \bar{\mathcal X}$.
Also
\begin{align*}
y_i = (1-\lambda)x_i^* + \lambda x_i
\end{align*}
implies that
\begin{align}
y_i &= 0 \qquad \qquad \mbox { for } i \in {\mathcal I}_0 \label{eq:ybounds0}\\
 \max\{x^*_i, x_i\} \ge y_i & \ge  x^*_i/2 >0	 \quad  \mbox { for } i \in {\mathcal I}_1	\label{eq:ybounds1}\\
  y_i  & =\lambda x_i>0 \quad  \mbox { for } i \in {\mathcal I}_2 \label{eq:ybounds2}
\end{align}
hence 
\[
f(y) = \sum_{i\notin {\mathcal I}_0}\gamma_i kl(y_i|\tilde{x}_i)\qquad f(x^*) = \sum_{i\notin {\mathcal I}_0}\gamma_i kl(x_i^*|\tilde{x}_i)
\]
By the mean value theorem, there exists  $\underline{\lambda} \in (0, \lambda)$ such that
\begin{align}
f(y)
&=f(x^*) + \sum_{i\notin {\mathcal I}_0}   \gamma_i\log \left( \frac{\underline{y}_i}{\tilde{x}_i }\right)(x_i - x^*_i)  \label{eq:f(T)} 
\end{align}
where $\underline{y} = (1-\underline{\lambda})x^* + \underline{\lambda} x$.
 
Consider any $i \in {\mathcal I}_1$.
It follows from \eqref{eq:ybounds1} that for all $\lambda \in (0, \frac{1}{2}]$,
 \[
\gamma_i\log \left( \frac{\underline{y}_i}{\tilde{x}_i }\right)(x_i- x^*_i) \le \beta_i
\]
  for some $\beta_i$.
Consider now any $i \in {\mathcal I}_2$.
It follows from \eqref{eq:ybounds2} that
\[
\lim_{\lambda \rightarrow 0} \; \underline{y}_i  = 0
\]
hence
\[
\lim_{\lambda \rightarrow 0} \gamma_i\log \left( \frac{\underline{y}_i}{\tilde{x}_i }\right)(x_i- x^*_i) = -\infty
\]
This implies  that, for $\lambda >0$ sufficiently small,
\[
\sum_{i\notin {\mathcal I}_0} \gamma_i\log \left( \frac{\underline{y}_i}{x_i }\right) (x_i - x^*_i) <0
\]
which along with \eqref{eq:f(T)} yields the contradiction that
$
f(y)  < f(x^*).
$ 
Hence, if  $x^*_i =0$  then $x_i =0$ for all   $x \in {\mathcal C} \cap\bar{\mathcal X}$.
\hfill
$\qed$
\end{proof}

The  general OT problem in \eqref{eq:GenOpt2} is equivalent to minimizing
\begin{align}
\label{eq:J(T,v)2}
\sum_{(i,j)\in \mathcal{Z}}  kl(t_{ij} | k_{ij})  + \sum_{l=1}^{N_1} 
\gamma_l kl(b^l | \tilde{b}^l )
\end{align}
subject to 
\begin{align}
\label{eq:GenConst3}
\sum_{(i,j) \in {\mathcal Z}} a^l_{ij} t_{ij} -b^l&= 0\qquad l=1,2,\dots, N_1\\
\label{eq:GenConst4}
\sum_{(i,j) \in \mathcal Z} a^l_{ij}t_{ij} &= \tilde{b}^l
\qquad l= N_1+1, N_1+2,\dots, N
\end{align}
and
\begin{align}
\label{eq:t>02}
 t_{ij} \ge 0 \mbox{ for } (i,j) \in {\mathcal Z},\quad 
b^l \ge 0  \mbox{ for } l=1,2, \cdots, N_1	
 \end{align}
The constraints in \eqref{eq:GenConst3}-\eqref{eq:t>02} describe a convex set and the 
function  in \eqref{eq:J(T,v)2} is of the form of the one in \eqref{eq:genCost1}. Hence the above lemma yields the following corollary.


\begin{corollary}
\label{cor:cor1}
Suppose $T^*$ is a minimizer for the general OT problem in  \eqref{eq:GenOpt2} and $t^*_{ij} =0$ for some $(i,j) \in {\mathcal Z}$.
Then $t_{ij} = 0$ for all $T$ satisfying \eqref{eq:GenConst}.
\end{corollary}

\subsection{ A characterization of an interior optimizer}
Consider the
problem of minimizing $f$ in \eqref{eq:genCost1}
subject to $x\in \bar{\mathcal X}$
and the constraints
\begin{align}
\label{eq:linConst}
\sum_{i=1}^q a^l_{i}x _i= \tilde{b}^l \quad l =1, 2, \dots, N
\end{align}
where $a^l_i, \tilde{b}^l \in \mathbb{R}$.

\begin{lemma}
\label{lem:lem2}
Suppose 
$\tilde{x}_i \in {\mathcal X}$, $\gamma_i >0$ for $i=1,2,\dots, q$ and 
 for each $i \in \{1,2, \dots, q\}$ there is an $x\in \mathbb{R}^q $ with   $x_i >0$  which satisfies  \eqref{eq:linConst}.
Then
$x^*$ minimizes 
\begin{equation*}
f(x)= \sum_{i=1}^q\gamma_i kl(x_i|\tilde{x}_i)
\end{equation*}
  subject to $x \in \bar{\mathcal X}$ and constraints \eqref{eq:linConst}
 if and only if 
 there exist positive  scalars $d_1,d_2, \dots, d_N$ 
such that
\begin{equation}
\label{eq:xstar}
x^*_i = \left(\prod_{ {\hat l}=1}^Nd_{\hat l}^{a^{\hat l}_{i}/\gamma_i} \right) \tilde{x}_i\qquad i=1,2, \dots, q
\end{equation}
and
\begin{align}
\label{eq:linConst*}
\sum_{i=1}^q a^l_i \left(\prod_{\hat{l}=1}^Nd_{\hat l}^{a^{\hat l}_i/\gamma_i} \right)\tilde{x}_i =\tilde{b}^l \qquad l=1,2,\dots, N
\end{align}

\end{lemma}

\begin{proof}
Since,  for each $i \in \{1,2, \dots, q\}$ there is an $x\in \mathbb{R}^q $ with   $x_i >0$
which satisfies \eqref{eq:linConst}, it follows from  Lemma \ref{lem:lem1} that
\[
x^*_i >0 \qquad i=1,2,\dots, q
\]
Hence $x^*$ is in the open set $ {\mathcal X}$.
The Lagrangian associated with this optimization  problem is
\begin{align*}
L(x, \alpha)  =    \sum_{i=1}^q  \gamma_i kl(x_i | \tilde{x}_i) 
+\sum_{\hat{l}=1}^N \alpha_{\hat{l}} \left(\sum_{i=1}^q a^{\hat{l}}_i x_i- \tilde{b}^{\hat l} \right)
\end{align*}
and $x^*$ is a minimizer if and only if
\begin{align}
\label{eq:linConst*2}
\sum_{i=1}^q a^l_ix^*_i= \tilde{b}_l \quad l=1, 2, \dots, N
\end{align}
and
 there exist scalars $\alpha_1,\alpha_2, \dots, \alpha_N$ such that for each $i=1,2,\dots, q$,
\begin{align*}
\frac{\partial L}{\partial x_i}(x^*, \alpha)  = 0
\end{align*}
that is,
\begin{align*}
\gamma_i \log \left( \frac{x^*_i}{\tilde{x}_i }  \right) +\sum_{{\hat l}=1}^N\alpha_{\hat l} a^{\hat l}_i =0
\end{align*}
or
\[
x^*_i = \tilde{x}_i e^{-\sum_{{\hat l}=1}^N\alpha_{\hat l} a^{\hat l}_i /\gamma_i} =
\left(\prod_{\hat{l}=1}^Nd_{\hat l}^{a^{\hat l}_i/\gamma_i} \right) \tilde{x}_i
\]
where
$
d_{\hat l}= e^{-\alpha_{\hat l}}.
$
Equality  \eqref{eq:linConst*} results from \eqref{eq:linConst*2}.

\end{proof}

\begin{lemma}
\label{lem:lem2bGen}
Suppose that for all $(i,j) \in {\mathcal Z}$ there is a matrix $T\in {\mathcal T}$ satisfying \eqref{eq:GenConst} with $t_{ij} >0$.
Then $T^*$ is a minimizer for the general OT problem in  \eqref{eq:GenOpt2} iff
there exist positive scalars $d_1,d_2, \dots, d_N$ 
 such that
\begin{equation}
\label{eq:tijstar}
t^*_{ij} = \left(\prod_{\hat{l}=1}^Nd_{\hat l }^{a^{\hat l}_{ij}} \right)
k_{ij}\qquad 
(i, j)\in \mathcal{Z}
\end{equation}
and
\begin{align}
\label{eq:dconst3}
\sum_{(i,j) \in \mathcal{Z}} a^l_{ij}
\left(\prod_{{\hat l }=1}^Nd_{\hat l }^{a^{\hat l}_{ij}} \right) k_{ij} 
&=d_{l}^{-1/\gamma_{l}} \tilde{b}_l
\quad l= 1,2, \cdots, N_1
\\
\sum_{(i,j) \in \mathcal{Z}} a^l_{ij}\left(\prod_{{\hat l}=1}^Nd_{\hat l}^{a^{\hat l}_{ij}} \right)
 k_{ij} 
&=  \tilde{b}_l
\qquad l=N_1+1, N_1+2, \cdots, N
\label{eq:dconst4}
\end{align}

\end{lemma}

\begin{proof}  
Recall that the optimization problem in  \eqref{eq:GenOpt2}  is equivalent to minimizing the function in 
\eqref{eq:J(T,v)2} subject to  \eqref{eq:GenConst3}-\eqref{eq:t>02}. 
This latter problem is of the form considered in Lemma \ref{lem:lem2}. It is assumed that for each $(i,j) \in {\mathcal Z}$ there is a matrix $T\in {\mathcal T}$ satisfying \eqref{eq:GenConst}  with $t_{ij} >0$.
Since the set of feasible transport maps is convex,
there is a feasible $T$ with $t_{ij} >0$ for all $(i,j) \in {\mathcal Z}$.
Since for each $l$ there exists $(i,j) \in {\mathcal Z}$ such that $a_{ij} >0$ n
it follows from constraint \eqref{eq:GenConst3} that there is a feasible $(T,b)$
with $t_{ij} >0$ for all $(i,j) \in {\mathcal Z}$ and  $b^l >0$ for $l=1,2,\cdots, N_1$ .

Let $d_{1}, d_{2}, \dots, d_N$  represent the positive scalars associated with constraints
\eqref{eq:GenConst3}  and \eqref{eq:GenConst4}.
The parameter $\gamma_i$  associated with each $t_{ij}$ equals one. It now follows from Lemma \ref{lem:lem2} and 
\eqref{eq:xstar} that
\[
t^*_{ij} = \left(\prod_{{\hat l}=1}^Nd_{\hat l}^{a^{\hat l}_{ij}} \right)
k_{ij}\qquad (i,j) \in \mathcal{Z}
\]
Each $b_l$ is only associated with   one constraint in \eqref{eq:GenConst3}. 
The  $a^l_i$ parameter associated with $b_l$  is $-1$ and the $\gamma_i$ parameter equals 
$\gamma_{l}$.
It now follows from Lemma \ref{lem:lem2} and 
\eqref{eq:xstar} that
\[
b^*_l = d_{l}^{-1/\gamma_{l}} \tilde{b}_l
\qquad l=1,2,\dots, N_1
\]

Constraint \eqref{eq:GenConst3} results in \eqref{eq:dconst3}
and constraint \eqref{eq:GenConst4} results in \eqref{eq:dconst4}.

\hfill
$\qed$
\end{proof}

Using the above lemma we now obtain a characterization of the optimizers for the OT problem 
in \eqref{eq:GenOpt}.

\begin{lemma}
\label{lem:lem2b}
Suppose that for all $(i,j) \in {\mathcal Z}$ there is a matrix $T\in {\mathcal T}$ satisfying \eqref{eq:rowconstraints1} and  \eqref{eq:colconstraints1} with $t_{ij} >0$.
Then $T^*$ is a minimizer for the problem in  \eqref{eq:GenOpt}   iff
there exist    positive scalars, $d_{11}, \dots,  d_{1m}$ and $d_{21},\dots, d_{2n}$,  such that, 
\begin{equation}
t^*_{ij} = d_{1i} d_{2j}k_{ij}   \qquad (i, j)
\in  \mathcal{Z}
\label{eq:TKscale2}
\end{equation}
and
\begin{align}
 d_{1i} &=\frac{ \tilde{u}_i} {\sum_{j=1}^n  k_{ij} d_{2j}}	\qquad \qquad \quad i\in {\mathcal I}
 \label{eq:d1i}\\
  d_{1i} &= \left( \frac{ \tilde{u}_i} {\sum_{j=1}^n  k_{ij} d_{2j}}\right)^{\frac{\gamma_{1i}}{1+\gamma_{1i}}}
  \qquad i \notin {\mathcal I}
  \label{eq:d1i2}\\
  d_{2j} &= \frac{ \tilde{v}_j}{\sum_{i=1}^m  d_{1i} k_{ij}}	
  \qquad \qquad \quad j\in {\mathcal J}	 \label{eq:d2j}\\
d_{2j} &= \left(\frac{ \tilde{v}_j}{\sum_{i=1}^m  d_{1i} k_{ij}}\right)^{\frac{\gamma_{2j}}{1+\gamma_{2j}}}
  \qquad j \notin {\mathcal J}
    \label{eq:d2j2}
\end{align}

\end{lemma}
 \begin{proof}  
 
Here we apply  Lemma \ref{lem:lem2bGen}.
Let $d_{1i}$ with $l\in \mathcal{I}$ 
be the positive scalar associated with   constraint $l=i$ in \eqref{eq:rowconstraints1}.
This is a special case of \eqref{eq:GenConst} with
\begin{equation}
\label{eq:bl}
\tilde{b}^l = \tilde{u}_l
\end{equation}
and
\begin{equation}
\label{eq:alij}
a^l_{ij} = \left\{ \begin{array}{rcl}
1 &\text{if} & i=l  \\
0 &\text{if} & i \neq l  \\
\end{array} \right.
\end{equation}

Let $d_{1l}$ with $l \notin \mathcal{I}$ 
be the positive scalar associated with the cost term
$\gamma_{1l} kl\left(\left. \sum_{j=1}^n  t_{lj} \right| \tilde{u}_l\right)$.
This is a special case of 
$\gamma_l kl\left( \sum_{(i,j) \in {\mathcal Z}} 
\left. a^l_{ij} t_{ij}\right|\tilde{b}^l  \right)$
with
$\tilde{b}^l$ and $a^l_{ij}$ given by \eqref{eq:bl} and \eqref{eq:alij}, respectively.

Similarily, let  $d_{2l}$ with $l\in \mathcal{J}$ 
be  the positive scalar associated with $j=l$ in constraints  \eqref{eq:colconstraints1}.
This is a special case of \eqref{eq:GenConst} with
\begin{equation}
\label{eq:bl2}
\tilde{b}^l = \tilde{v}_l
\end{equation}
and
\begin{equation}
\label{eq:alij2}
a^l_{ij} = \left\{ \begin{array}{rcl}
1 &\text{if} & j=l  \\
0 &\text{if} & j \neq l  \\
\end{array} \right.
\end{equation}

Let $d_{2l}$ with $l\notin \mathcal{J}$ 
be the positive scalar associated with the cost term
$\gamma_{2l} kl\left(\left. \sum_{i=1}^m  t_{il} \right| \tilde{v}_l\right)$.
This is a special case of 
$\gamma_l kl\left( \sum_{(i,j) \in {\mathcal Z}} 
\left. a^l_{ij} t_{ij}\right|\tilde{b}^l  \right)$
with
$\tilde{b}^l$ and $a^l_{ij}$ given by \eqref{eq:bl2} and \eqref{eq:alij2}, respectively.

It now follows from \eqref{eq:tijstar}, \eqref{eq:alij} and \eqref{eq:alij2} that
\[
t^*_{ij} = d_{1i}d_{2j} k_{ij} \qquad (i, j) \in \mathcal{Z}
\]

Recalling \eqref{eq:dconst4}, \eqref{eq:bl} and \eqref{eq:alij}, for each  $i \in \mathcal{I}$, constraint \eqref{eq:colconstraints1} results in
\[
\sum_{j=1}^n d_{1i}d_{2j}k_{ij} =\tilde{u}_i
\]
(Recall that $k_{ij} =0$ when $(i,j)\notin \mathcal{Z}$).
This 
 results in 
\eqref{eq:d1i}.

Recalling \eqref{eq:dconst3},\eqref{eq:bl} and \eqref{eq:alij}, for each  $i \notin \mathcal{I}$, 
the cost term $
\gamma_{ij} kl\left(\left. \sum_{j=1}^n  t_{ij} \right| \tilde{u}_i\right) $
results in
\[
\sum_{j=1}^n d_{1i}d_{2j}k_{ij} =d_{1i}^{-1/\gamma_i}\tilde{u}_i
\]
which implies 
\eqref{eq:d1i2}.

In a similar fashion, consideration of constraints \eqref{eq:colconstraints1} and cost terms $\gamma_{2j} kl\left(\left. \sum_{i=1}^m  t_{ij} \right| \tilde{v}_j\right) $
results in 
\eqref{eq:d2j} and \eqref{eq:d2j2}, respectively.

\hfill
$\qed$
\end{proof}

\begin{remark}
Note that, if $t^*_{ij} = 0$ for some  $(i, j) \in \mathcal{Z}$ then, 
\eqref{eq:TKscale2} no longer holds.
In this case, $t_{ij} = 0$ for all $T$ satisfying  
\eqref{eq:rowconstraints1} and  \eqref{eq:colconstraints1}; 
recall Corollary \ref{cor:cor1}.
\end{remark}

\section{Convergence of the proposed algorithms}
\label{sec:Proof}

Recall the general optimization problem
\begin{equation}
\label{eq:opt0}
\min_{x \in \bar{\mathcal X}}
\sum_{i=1}^q \gamma_i kl(x_i|\tilde{x}_i)
\quad \mbox{s.t.} \quad 
\sum_{i=1}^q a^l_{i}x _i= \tilde{b}^l \quad l =1, 2, \dots, N
\end{equation}
where $\tilde{x}_i >0$ and $\gamma_i >0$ for $i=1,2, \dots, q$.

\begin{assumption}
\label{ass1}
For each 
$l=1,2,\dots, N$ and $\hat{x}\in\mathcal{X}$,  there exists  
$x^*\in \mathcal{X}$
which solves the optimization problem
\begin{align}
\label{eq:opt0l}
\min_{x \in \bar{\mathcal X} } \sum_{i=1}^q \gamma_i kl(x_i|\hat{x}_i) \qquad s.t.\quad  \sum_{i=1}^q a^l_{i}x _i= \tilde{b}^l 
\end{align}
\end{assumption}

Note that, in this assumption,  the minimizer is in  \mbox{$\mathcal{X}$}.
We will denote the point $x^*$ above by 
$
P_l(\hat{x})
$.
The  following result may be gleaned from  \cite{Bregman1967}.

\begin{theorem}
\label{th:Bregman}
Suppose the optimization problem \eqref{eq:opt0} is feasible,  Assumption \ref{ass1} holds and
\begin{align}
x(k+1) &= P_{l}(x(k))		\quad \text{with} \quad x(0)=\tilde{x}
\end{align}
where 
\begin{equation*}
l= l_k= k\, \bmod N +1
\end{equation*}
Then, the sequence $\{x(k)\}$ converges to a limit $x^* \in \bar{\mathcal X}$ and  $x^*$
is a minimizer for  optimization problem \eqref{eq:opt0}.
\end{theorem}

This algorithm is very useful when one can readily solve the optimization problems in \eqref{eq:opt0l}.
The following lemma characterizes optimizers for \eqref{eq:opt0l}.

\begin{lemma}
\label{lem:lembasic}
Suppose $\hat{x} \in \mathcal{X}$.
Then
$
x^*  \in \mathcal{X}
$
minimizes
\begin{equation}
\label{eq:gencost}
 \sum_{i =1}^q  \gamma_ikl(x_i | \hat{x}_i) 
 \end{equation}
 subject to $x\in \bar{\mathcal X}$ and
\begin{equation}
\label{eq:linconstraint}
\sum_{i=1}^q a_ix_i = \tilde{b}
 \end{equation}
iff there exists $c>0$ such that
\begin{equation}
\label{eq:lem3result}
\begin{aligned}
x^*_i &= c^{a_i/\gamma_i}  \hat{x}_i 
\qquad i=1,2,\dots, q
\end{aligned}
\end{equation}
and
\begin{equation}
\label{eq:c1i0}
 {\sum_{i=1}^q a_{i}\hat{x}_ic^{a_i/\gamma_i}} = \tilde{b}
\end{equation}
\end{lemma}

\begin{proof}
The Lagrangian $L$  associated with this optimization  problem is 
\begin{align*}
 L(x,\alpha)=
  \sum_{i =1}^q  \gamma_ikl(x_i| \hat{x}_i) 
+ \alpha \left(\sum_{i=1}^q  a_ix_i -\tilde{b} \right)
\end{align*}
where $\alpha \in \mathbb{R}$.
Since ${\mathcal X}$ is open, $x^* \in {\mathcal X}$ is a minimizer 
iff
\begin{equation}
\label{eq:linconst2}
 \sum_{i=1}^q a_i x^*_i = b
\end{equation}
and there is a scalar $\alpha$ such that, for $i=1,2,\dots, q$,
\[
\frac{\partial L}{\partial x_i} (x^*_i, \alpha)  =0 
\]
that is,
\begin{align*}
     \gamma_i \log \left( \frac{x^*_i}{\hat{x}_i}  \right)  +\alpha a_i  =0 
\end{align*}
or,
\[
x^*_i = \hat{x}_i e^{-\alpha a_i/\gamma_i} =c^{a_i/\gamma_i} \hat{x}_i 
\]
where
\[
 c = \exp(   -\alpha )
\]
Also, \eqref{eq:linconst2} 
results in \eqref{eq:c1i0}.
\hfill$\qed$
\end{proof}

\begin{remark}
\label{rem:f(c) = 0}
In special cases \eqref{eq:c1i0} can readily be solved for $c$.
Consider the situation in which
\[
a_i/\gamma_i = \kappa \qquad i=1,2, \dots, q
\]
In this case, \eqref{eq:c1i0} reduces to
\[
\sum_{i=1}^q a_i \hat{x}_ic^\kappa  = \tilde{b}
\]
hence
\[
c^\kappa = \frac{\tilde b}{\sum_{i=1}^q a_{i} \hat{x}_i}
\]
and
$x^*_i = c^\kappa \hat{x}_i$.

Consider also the case in which
\[
\tilde{b}= 0 \qquad a_i/\gamma_i = \kappa \qquad i=1,2, \dots, q-1
\]
In this case, \eqref{eq:c1i0} reduces to
\[
\sum_{i=1}^{q-1} a_{i} \hat{x}_ic^\kappa  + a_q \hat{x}_q c^{a_q/\gamma_q}  = 0
\]
hence
\[
c^{\kappa -a_q/\gamma_q} = \frac{-a_q\hat{x}_q}{\sum_{i=1}^{q-1} a_{i} \hat{x}_i}
\]
$x^*_i = c^\kappa \hat{x}_i$ for $i=1,2, \cdots, q-1$
and $x^*_q = c^{a_q/\gamma_q} \hat{x}_q$.

In the general case, one needs to solve
\begin{equation}
\label{eq:f(c)=0}
f(c) =0
\end{equation}
where
\[
f(c) =  {\sum_{i=1}^q a_{i}\hat{x}_i}c^{a_i/\gamma_i} -\tilde{b}
\]
If $a_i,\gamma_i,  >0$ for $i=1,2, \cdots, q$ and $\tilde{b}>0$ then, 
\[
f(0) =  -\tilde{b} <0 \qquad 
\lim_{c\rightarrow \infty} f(c) = \infty
\]
and
\[
f'(c)  = {\sum_{i=1}^q  \frac{a_i^2}{ \gamma_i} \hat{x}_i} c^{a_i/\gamma_i\!-\!1} >0 \qquad \text{for } c>0
\]
Hence, \eqref{eq:f(c)=0} has a unique solution for $c$.

\end{remark}


 Using Theorem \ref{th:Bregman} and Lemma \ref{lem:lembasic} we obtain the following  algorithm to solve
 optimization problem \eqref{eq:opt0}.


{\bf Algorithm 5.}
Initialize with
\begin{equation}
\label{eq:x0}
x(0)=\tilde{x}
\end{equation}
Iterate for $k=0,1,\dots $:
 
With
$
l= l_k :=  k\text{mod}\, N +1$, solve 
\begin{equation}
\label{eq:c1i2}
 \sum_{i=1}^q a^{l}_i x_i(k)  c_{l}(k+1)^{a^l_i/\gamma_i}= \tilde{b}^{l}
\end{equation}
for $c_l(k+1)$
and let
\begin{align}
\label{eq:genAl1}
x_i(k+1) =  c_{l}(k+1)^{a^l_i/\gamma_i} x_i(k) 	\quad i=1,2,\dots, q	
\end{align}

\begin{theorem}
\label{th:Genal}
Suppose the optimization problem \eqref{eq:opt0} is feasible.
Then, the sequence $\{x(k)\}$generated by Algorithm 5  converges to a limit $x^* \in \bar{\mathcal X}$ and  $x^*$
is a minimizer for  optimization problem \eqref{eq:opt0}.
\end{theorem}

Introducing  new variables defined by
\begin{align*}
d_l(0) &= 1 	\qquad \qquad \qquad l=1,2,\dots, N\\
\quad  d_l(k+1) &=c_l(k+1)d_l(k) \quad  l=l_k  \\
\quad  d_l(k+1) &=d_l(k)  \, \, \qquad \qquad l=1,2,\dots, N,  l\neq l_k
\end{align*}
we obtain another algorithm for solving  optimization problem \eqref{eq:opt0}.

{\bf Algorithm 6}
Initialize with
 \[
d_l(0) = 1\qquad l=1,2,\dots, N 
\]
Iterate for $k=0,1, \cdots$:

With
$
l_k =  k\, \text{mod}\, N +1
$,
let
\begin{equation}
\label{eq:dlkp1}
d_l(k+1) =d_l(k)  \qquad  l=1,2,\dots, N \qquad  l\neq l_k
    \end{equation}
    and
solve 
\begin{equation}
\label{eq:dlGen}
 \sum_{i=1}^q {a^{l_k}_{i}}\tilde{x}_i   \prod_{l=1}^Nd_{l}(k+1)^{a^l_{i}/\gamma_i}
   = \tilde{b}^{l_k}
 \end{equation}
 for $d_{l_k}(k+1)$,
and let $x(k+1)$ be given by 
\begin{equation}
\label{eq:xi(k+1)}
x_i(k+1) = \prod_{l=1}^Nd_{l}(k+1)^{a^l_{i}/\gamma_i}\, \tilde{x}_i
	\quad i=1,2,\dots, q
\end{equation}

Equation  \eqref{eq:xi(k+1)} follows from \eqref{eq:x0}, \eqref{eq:genAl1}, \eqref{eq:dlkp1}  and \eqref{eq:dlGen} follows from \eqref{eq:c1i2}.

\subsection{Proof of Theorem \ref{th:Genal1}}
\label{sec:prth1}

Recall that the general OT optimization problem in  \eqref{eq:GenOpt2}  is equivalent to minimizing the function in 
\eqref{eq:J(T,v)2} subject to \eqref{eq:GenConst3}-\eqref{eq:t>02}.
 Also,  constraints \eqref{eq:GenConst3}-\eqref{eq:GenConst4} can be expressed as
$\sum_{i=1}^qa^l_{i}x_i = \tilde{b}^l$ for $l=1,2,\dots, N$.
 So, we apply the algorithm in Theorem \ref{th:Genal}.

Considering  constraint $l$ in \eqref{eq:GenConst3},
\eqref{eq:genAl1} and \eqref{eq:c1i2}
result in
$$
\begin{aligned}
t_{ij}(k+1) &= c_{l}(k+1)^{a^l_{ij}}  t_{ij}(k) \quad (i, j) \in {\mathcal Z}\\
b^l(k+1) &= c_{l}(k+1)^{-1/\gamma_l} b^l(k)
\end{aligned}
$$
the variables $b^1, b^2, \dots,  b^{l-1}, b^{l+1}, \dots,b^{N_1}$ remain unchanged
and
\begin{equation*}
 \sum_{(i,j) \in {\mathcal Z}} a^l_{ij} t_{ij}(k)c_{l}(k+1)^{a^l_{ij}}   -b^l(k)c_{l}(k+1)^{-1/\gamma_l} =0
 \end{equation*}
 
Considering  constraint $l$ in \eqref{eq:GenConst4},
\eqref{eq:genAl1} and \eqref{eq:c1i2}
result in
$$ 
t_{ij}(k+1) = c_{l}(k+1)^{a^l_{ij}}  t_{ij}(k) \quad (i, j) \in {\mathcal Z}
$$
the variables $b^1, b^2, \dots,b^{N_1}$ remain unchanged
and
\begin{equation*}
 \sum_{(i,j) \in {\mathcal Z}} a^l_{ij} t_{ij}(k) c_{l}(k+1)^{a^l_{ij}}   = \tilde{b}^l
\end{equation*}
$\qed$

\subsection{Obtaining Algorithm 3 from Algorithm 1}
\label{sec:3from1}

For $i\in \mathcal{I}$, let $c_{1i}$ 
be the
scalar associated with   constraint $i$ in \eqref{eq:rowconstraints1}.
This is a special case of \eqref{eq:GenConst} with
\begin{equation}
\label{eq:bl2}
\tilde{b}^l = \tilde{u}_i
\end{equation}
and
\begin{equation}
\label{eq:alij2}
a^l_{ij} = \left\{ \begin{array}{rcl}
1 &\text{if} & i=l  \\
0 &\text{if} & i \neq l  \\
\end{array} \right.
\end{equation}
So,
\eqref{eq:tij(k+1)2} and \eqref{eq:cl(k+1)2}  result in
\begin{equation}
\label{eq:lem3result2}
t_{ij}(k+1) = c_{1i}(k+1)t_{ij}(k) \qquad (i, j) \in {\mathcal Z} 
\end{equation}
and
\begin{equation*}
\sum_{j=1}^n t_{ij}(k) c_{1i}(k+1) = \tilde{u}_i
\end{equation*}
respectively,
and
all other variables remain unchanged.
Hence
\begin{equation}
\label{eq:c1i02}
 c_{1i}(k+1)= \frac{\tilde{u}_i(k)}{\sum_{j=1}^n t_{ij}(k)}
\end{equation}

Similarily, 
considering  a column constraint  in \eqref{eq:colconstraints1} for $j \in \mathcal{J}$,
\eqref{eq:tij(k+1)2} and \eqref{eq:cl(k+1)2}  result in
 \begin{equation}
\label{eq:lem5result}
\begin{aligned}
t_{ij}(k+1) &= c_{2j}(k+1)  t_{ij}(k) \quad (i, j) \in {\mathcal Z}\\
\end{aligned}
\end{equation}
all other variables remain unchanged and
\begin{equation}
\label{eq:c2j1}
 c_{2j}(k+1)= \frac{\tilde{v}_j(k)}{\sum_{i=1}^m t_{ij}(k)}
\end{equation}

For  $i \notin \mathcal{I}$ let $c_{1i}$
be the 
scalar associated with the cost term
$\gamma_{1i} kl\left(\left. \sum_{j=1}^n  t_{ij} \right| \tilde{u}_i\right)$.
This is a special case of 
$\gamma_l kl\left( \sum_{(i,j) \in {\mathcal Z}} 
\left. a^l_{ij} t_{ij}\right|\tilde{b}^l  \right)$
with
$\tilde{b}^l$ and $a^l_{ij}$ given by \eqref{eq:bl2} and \eqref{eq:alij2}, respectively
So,
\eqref{eq:tij(k+1)1}  results in
\begin{equation}
\label{eq:lem4result}
\begin{aligned}
t_{ij}(k+1) &= c_{1i}(k+1)  t_{ij}(k) \qquad (i, j) \in {\mathcal Z}\\
u_i(k+1) &= c_{1i}(k+1)^{-\frac{1}{\gamma_{1i}} }u_i(k)	\\
\end{aligned}
\end{equation}
all other variables remain unchanged and  \eqref{eq:cl(k+1)1} implies that
\begin{equation*}
 \sum_{j=1}^nc_{1i}(k+1)t_{ij}(k) - c_{1i}(k+1)^{-1/\gamma_{1i} }u_i(k) =0
\end{equation*}
Hence
\begin{align}
\label{eq:c2j0}
c_{1i}(k+1)= \left(\frac{u_i(k)}{\sum_{j=1}^n t_{ij}(k)}\right)^\frac{\gamma_{1i}}{1+\gamma_{1i} }
\end{align}


%
%
Similarly, considering 
the cost term
$\gamma_{2j} kl\left(\left. \sum_{i=1}^m  t_{ij} \right| \tilde{v}_j\right)$ for $j \in \mathcal{J}$
\eqref{eq:tij(k+1)1} and \eqref{eq:cl(k+1)1}
result in
\begin{equation}
\label{eq:lem6result}
\begin{aligned}
t_{ij}(k+1) &= c_{2j}(k+1)  t_{ij}(k) \qquad (i,j) \in {\mathcal Z}\\
v_j(k+1) &= c_{2j}(k+1)^{-\frac{1}{\gamma_{2j}} }v_j(k)	\\
\end{aligned}
\end{equation}
all other variables remain unchanged and
\begin{align}
\label{eq:c2j2}
c_{2j}(k+1)= \left(\frac{v_j(k)}{\sum_{i=1}^m t_{ij}(k)}\right)^\frac{\gamma_{2j}}{1+\gamma_{2j} }
\end{align}

Algorithm 1 is obtained by performing \eqref{eq:lem3result2}-\eqref{eq:c2j2} in one step.

\section{Use-cases and applications}\label{sec:applications}
To illustrate the utility of our results we provide examples from two distinct domains: an example from the gig economy for the reallocation of resources under fairness constraints, and a resource-allocation example from energy markets designed to capitalise on consumers' flexibility.

\subsection{Equalising Utility across Groups of Drivers}
\begin{figure}
    \centering
    \includegraphics[width=0.85\linewidth]{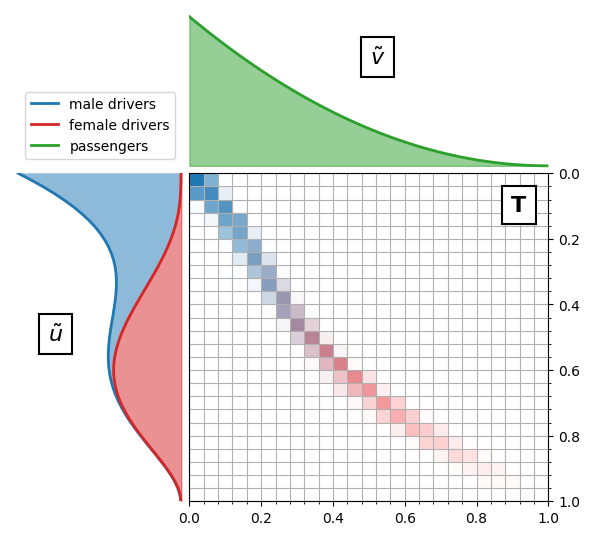}
    \caption{Consider the distribution of \textcolor{blue}{male} and \textcolor{red}{female} drivers in $\tilde{u}$, where male drivers are more likely to be close to the high-demand city centre ($x=0.0$). Naively allocating passengers to drivers in this setting would proliferate the gender pay gap. Instead, we constrain the allocation plan $T$ such that the subgroups have \textit{equal earning opportunity}.}
    \label{fig:problem_statement}
\end{figure}

We consider the problem, wherein taxi drivers must be dispatched from a set of city locations to meet passenger demand. Despite never discriminating directly based on drivers' gender, there remains an intrinsic pay gap based on factors including drivers' home locations \cite{cook2021gender}. Male drivers are more likely to be based closer to the city centre, and other lucrative high-crime neighbourhoods, whereas female drivers are more frequently based in quieter suburbs as shown in Figure \ref{fig:problem_statement}.

We model this problem in an OT framework where a distribution of drivers must be matched to a distribution of passengers. We simulate the distribution of male and female drivers as beta distributions $\tilde{u}_{i,\text{male}} \sim \mathcal{B}(1.0, 5.0)$ and $\tilde{u}_{i,\text{female}} \sim \mathcal{B}(4.0, 3.0)$ respectively, and the passenger demand is modelled $\tilde{v}_j \sim \mathcal{B}(1.0, 3.0)$ to reflect increased demand in the city centre. The proportion of female drivers at each aggregate location $x_i$ is therefore 
$w_i \equiv \frac{\tilde{u}_{i,\text{female}}}{\tilde{u}_{i}}$.

We define the transport cost $c_{ij}$ as the squared Euclidean distance from $x_i$ to $x_j$ to approximate a matching based on the nearest driver-passenger pairs, and regularise with $\gamma_0 \equiv 0.001$. We set $\gamma_{1,i} = \gamma_{2,j} = 10 $ for all $(i,j)$ to penalise unmet demand and under-utilised supply as a result of meeting the utility constraint.

If $s_j \in [5,20]$ is an   estimate of the mean fare from  pickup point
$j$ we can calculate the estimated utility for all male drivers and all female drivers and set them to be equal:

\begin{equation*}  
\sum_i w_i \sum_j s_{j} t_{ij} = \sum_i (1-w_i) \sum_j s_{j} t_{ij}
\end{equation*}
where $t_{ij}$ is the number of passengers from block $x_j$ transported by drivers from block $x_i$. We can rearrange this constraint into the form of (\ref{eq:GenConst}):
\begin{equation}
    \sum_{i} \sum_{j}(2 w_i - 1) s_j t_{ij} = 0
    \label{eq:fairness_constraint}
\end{equation}

Normalised log-convergence for the constrained algorithm is shown in Figure \ref{fig:convergence_matching}. Our method is able to effectively and efficiently trade off optimality against realising the constraint: without the utility constraint, the pay gap is as high as 22.5\%, but introducing constraint \eqref{eq:fairness_constraint} reduces this to effectively zero for a global cost of only 3.4\% higher.


\begin{figure}[h]
    \centering
    \includegraphics[width=0.75\linewidth]{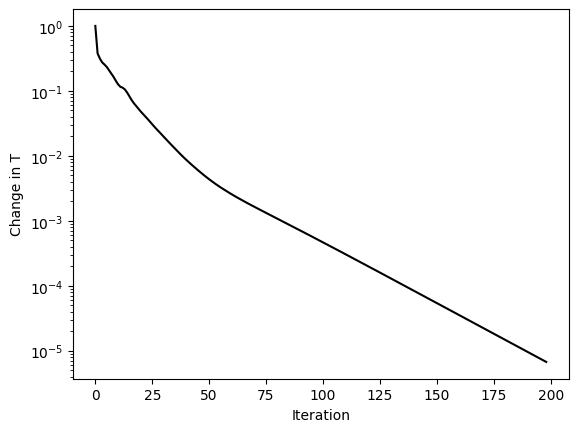}
    \caption{Convergence of $T(k)$ over 200 iterations of Algorithm 2}
    \label{fig:convergence_matching}
\end{figure}

\subsection{Exploiting Flexibility in Resource Allocation Problems}

Next, consider an energy market setting where we have some finite supply from different sources, and a total demand which exceeds the total supply ($\sum_i \tilde{u}_i < \sum_j \tilde{v}_j$). Suppose that some of our consumers \textit{must} have their demand exactly realised, but that others have some flexibility to be under- or over-supplied. Each flexible consumer has some \textit{flexibility cost}, i.e. the reward they have negotiated for each unit of flexibility that is realised under the optimal plan $T^*$. Clearly, the system is incentivised to minimise the total cost of the realised flexibility required by the plan. By setting the column penalty $\gamma_{2j}$ proportional to the cost of agent $j$'s flexibility, the objective reflects both the transport cost $\sum_{(i,j) \in \mathcal{Z}} c_{ij} t_{ij}$ and the flexibility cost $\sum_{j \notin \mathcal{J}} \gamma_{2j} \;kl (\sum_i t_{ij} | \tilde{v}_j)$.

To demonstrate this, we consider  $n=200$ suppliers whose capacities $\tilde{u}_i$ are  drawn iid from the normal distribution $\mathcal{N}(12.5, 2^2)$.
Additionally, we consider  $m=500$ consumers, whose ideal demands $\tilde{v}_j$
are  drawn iid from the normal distribution $\mathcal{N}(5.0, 1^2)$. The suppliers are completely inflexible, therefore $ \mathcal{I}  = \{1, 2, \cdots, n\}$. Of the consumers, 25\% are inflexible in their demand (consider, for example, factories or hospitals), and the remaining 75\% are assigned an iid flexibility cost $\gamma_{2j} \sim \mathcal{U}(2.5, 50.0)$. Additionally, we randomly draw $n+m = 700$ source-target pairs $(i, j) \notin \mathcal{Z}$ to simulate some infrastructure or delivery constraints. We set the regularisation parameter $\gamma_0$ to $0.01$, and the transport cost is drawn from the uniform distribution $[0,1]$.

\begin{figure}[h]
    \centering
    \includegraphics[width=0.8\linewidth]{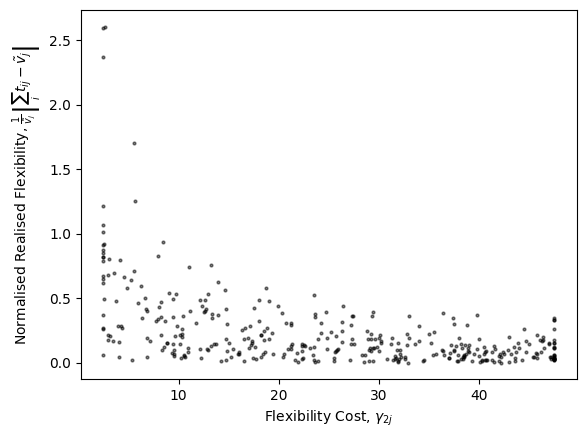}
    \caption{Realised flexibility $ \frac{1}{\tilde{v}_j} \left| \sum_i t_{ij} - \tilde{v}_j \right|$ against flexibility cost for consumers $j \notin \mathcal{J}$.}
    \label{fig:flexibility}
\end{figure}

Algorithm 4 is then applied until the plan $T(k)$ reaches convergence tolerance $1\times 10^{-12}$. The relationship between the flexibility cost, $\gamma_{2j}$, and the realised flexibility for each flexible consumer is shown in Figure \ref{fig:flexibility}: note that the maximum realised flexibility is inversely proportional to the flexibility cost, and as such the total cost of the flexibility in the plan is minimised. This demonstrates the algorithm's natural versatility to be applied to many different classes of problems.

\section{Conclusions}
\label{sec:conc}
We have presented a provably-convergent approach to solving optimal transport problems with linear constraints and selectively-relaxed marginal constraints. To the best of our knowledge, this is the first work which approaches the problem of selective constraint relaxation and provides convergence results. We demonstrate our algorithm's flexibility by applying it to two highly relevant problems in energy markets and the gig economy.

\bibliographystyle{abbrv}
\small
\bibliography{bibliography}



\end{document}